\documentclass{amsart}

\usepackage{amsfonts}
\usepackage{amssymb}
\usepackage{amsthm}
\usepackage{eucal}
\usepackage{tikz}

\theoremstyle{plain}
\newtheorem{theorem}{Theorem}[section]

\newtheorem{lemma}[theorem]{Lemma}
\newtheorem{corollary}[theorem]{Corollary}

\theoremstyle{definition}
\newtheorem*{definition}{Definition}

\theoremstyle{remark}
\newtheorem*{remark}{Remark}

\newcommand{\p}{\mathcal{P}}
\newcommand{\gen}[1]{\langle #1\rangle}
\newcommand{\Z}{\mathbb{Z}}

\newcommand{\Aut}{\mathrm{Aut}}
\newcommand{\normal}{\trianglelefteq}

\begin{document}
\title[Power graphs]{On power graphs of finite groups with forbidden induced subgraphs}
\author{A. Doostabadi, A. Erfanian and M. Farrokhi D. G.}
\keywords{Power graph, claw, star, cycle, forbidden subgraph}
\subjclass[2000]{Primary 05C25; Secondary 05C99.}
\address{Department of Pure Mathematics, Ferdowsi University of Mashhad, Mashhad, Iran}
\email{a.doostabadi@yahoo.com}
\address{Department of Pure Mathematics, Ferdowsi University of Mashhad, Mashhad, Iran}
\email{erfanian@math.um.ac.ir}
\address{Department of Pure Mathematics, Ferdowsi University of Mashhad, Mashhad, Iran}
\email{m.farrokhi.d.g@gmail.com}
\date{}
\begin{abstract}
The power graph $\p(G)$ of a finite group $G$ is a graph whose vertex set is the group $G$ and distinct elements $x,y\in G$ are adjacent if one is a power of the other, that is, $x$ and $y$ are adjacent if $x\in\gen{y}$ or $y\in\gen{x}$. We characterize all finite groups $G$ whose power graphs are claw-free, $K_{1,4}$-free or $C_4$-free. 
\end{abstract} 
\maketitle
\section{Introduction}
The \textit{power graph} $\p(G)$ of a group $G$ is a graph with elements of $G$ as its vertices such that two distinct elements $x$ and $y$ are adjacent if $y=x^m$ or $x=y^m$ for some positive integer $m$. Clearly, for finite (torsion) groups two distinct elements $x$ and $y$ are adjacent if and only if $x\in\gen{y}$ or $y\in\gen{x}$.

Power graphs of groups were brought up by Kelarev and Quinn \cite{avk-sjq:2000,avk-sjq:2002}. Subsequently Chakrabarty, Ghosh and Sen \cite{ic-sg-mks:2009} studied power graphs that are complete or Eulerian or Hamiltonian. Recently Cameron \cite{pjc:2010} has shown that two finite groups with isomorphic power graphs have the same number of elements of equal order. As to a converse statement, two finite abelian groups with isomorphic power graphs are isomorphic, see Cameron and Gosh \cite{pjc-sg:2011}.

A graph is said to be \textit{$\Gamma$-free} for some graph $\Gamma$ if it has no induced subgraphs isomorphic to $\Gamma$. Graphs with forbidden structures appear in many contexts like extermal graph theory where lower and upper bounds can be obtained for various numberical invariants of the corresponding graphs.

In this paper, we shall study power graphs of finite groups with forbidden subgraphs, that is, $\Gamma$-free graphs for some graphs $\Gamma$. Indeed, in Theorems 2.2 and 2.6, we will classify all finite groups whose power graphs have no induced subgraphs isomorphic to $K_{1,3}$ or $K_{1,4}$, respectively. Also, in Theorems 3.5--3.8, we give some structural results for finite groups whose power graphs have no induced subgraphs isomorphic to $C_4$, the cycle of length four. We remark that a group with $K_{1,1}$-free power graph is just the trivial group, a $K_{1,2}$-free power graph must be complete so that the correspoding group is a cyclic $p$-group, and a group with a triangle-free power graph is an elementary abelian $2$-group.

In what follows, $S_p(G)$, $\exp(G)$, $\pi(G)$ and $\omega(G)$ stand for a Sylow $p$-subgroup of $G$, the exponent of $G$, the set of all prime divisors of $|G|$ and the set of all orders of elements of $G$, respectively. Also, if $x$ and $y$ are vertices of a graph, then $x\sim y$ indicates that $x$ and $y$ are adjacent. A subset $S$ of vertices of a graph $\Gamma$ is said to be \textit{independent} if the subgraph induced by $S$ is an empty graph (graph with no edges). The maximum size of independent sets of $\Gamma$ is called the \textit{independence number} of $\Gamma$ and it is denoted by $\alpha(\Gamma)$. 
\section{Claw-free and $K_{1,4}$-free power graphs}
In this section, the structure of finite groups whose power graphs are claw-free ($K_{1,3}$-free) or $K_{1,4}$-free will be studied. Keep in mind that the $p$-rank $m_p(G)$ of a finite group $G$ is the maximum rank of an elementary abelian $p$-subgroup of $G$, where $p$ is a prime. Suppose for the moment that the $p$-rank of the finite $p$-group $G$ equals $1$.Then it is known that $G$ is a cyclic $p$-group or a generalized quaternion $2$-group, see \cite[5.3.6]{djsr:1996}.

We begin with the following simple lemma, which gives a necessary and sufficient condition for a finite $p$-group to admit a claw-free power graph.
\begin{lemma}\label{clawfreepgroups}
If $G$ is a finite $p$-group, then $\p(G)$ is claw-free if and only if $G$ is cyclic.
\end{lemma}
\begin{proof}
Let $r=m_p(G)$. If $r>1$ then $G$ has an elementary abelian $p$-subgroup $H=\gen{a_1,a_2,\ldots, a_r}$ of order $p^r$. Hence, the set $\{1,a_1,a_2,a_1a_2\}$ induces a claw, which is a contradiction. Thus $r=1$ and either $G$ is a cyclic $p$-group or $G$ is a generalized quaternion $2$-group. If $G=\gen{a,b:a^{2^{n-1}}=1, a^{2^{n-2}}=b^2, b^{-1}ab=b}$ is a generalized quaternion $2$-group, then again the vertices $1,a,b,ab$ induce a claw, which is a contradiction. Hence $G$ must be a cyclic $p$-group. The converse follows from the fact that cyclic $p$-groups have complete power graphs.
\end{proof}

Utilizing the above lemma we have the following characterization of claw-free power graphs.
\begin{theorem}\label{clawfreegroups}
If $G$ is a finite group, then $\p(G)$ is claw-free if and only if $G$ is a cyclic group of order $p^mq^n$, where $\{m,n\}\cap\{0,1\}\neq\emptyset$.
\end{theorem}
\begin{proof}
If $|\pi(G)|\geq3$, then clearly we can construct a claw in $\p(G)$. Thus $|\pi(G)|\leq2$. If $G$ is a $p$-group, then by Lemma \ref{clawfreepgroups}, $G$ is cyclic and we are done. Suppose $G$ is not a $p$-group. Then $|G|=p^m q^n$ for some distinct primes $p$ and $q$. By Lemma \ref{clawfreepgroups}, the Sylow $p$-subgroups and Sylow $q$-subgroups of $G$ are cyclic, hence $G=\gen{x}\rtimes\gen{y}$ by \cite[10.1.10]{djsr:1996}, where $\gen{x}$ and $\gen{y}$ are the Sylow $p$-subgroup and a Sylow $q$-subgroup of $G$, respectively. If $G$ is not cyclic, then $\gen{y}\neq\gen{y^g}$ for some $g\in G$, from which it follows that $\{1,x,y,y^g\}$ induces a claw in $\p(G)$, a contradiction. Therefore $G\cong\Z_{p^mq^n}$ is a cyclic group. If $m,n\geq2$ then the elements of orders $1,p^2,pq,q^2$ induce a claw in $\p(G)$, which is a contradiction. Thus either $m=0,1$ or $n=0,1$, as required. Conversely, a simple verification shows that $\alpha(\p(\Z_{p^m q^n}))=1$ or $2$, which implies that $\p(\Z_{p^m q^n})$ is claw-free. The proof is complete.
\end{proof}

To deal with groups with $K_{1,4}$-free power graphs, we need the following two theorems.
\begin{theorem}[\cite{mb-acb-am:1970,gs:1926}]\label{threecover}
Let $G$ be a finite group that is the union of three proper subgroups $H$, $K$ and $L$. Then $H\cap K=K\cap L=L\cap H$ and $G/H\cap K\cap L\cong \Z_2\times\Z_2$. In particular, $[G:H]=[G:K]=[G:L]=2$.
\end{theorem}
\begin{theorem}[{\cite[5.3.4]{djsr:1996}}] \label{maximalcycle}
A group of order $p^n$ has a cyclic maximal subgroup if and only if it is of one of the following types:
\begin{itemize}
\item[(i)]a cyclic group of order $p^n$,
\item[(ii)]the direct product of a cyclic group of order $p^{n-1}$ and one of order $p$, $n\geq 2$,
\item[(iii)]the modular $p$-group $M_{p^n}=\gen{a,b:a^{p^{n-1}}=b^p=1,a^b=a^{p^{n-2}+1}}$, $n\geq 3$,
\item[(iv)]the dihedral group $D_{2^n}=\gen{a,b:a^{2^{n-1}}=b^2=1,a^b=a^{-1}}$, $n\geq 3$,
\item[(v)]the generalized quaternion group $Q_{2^n}=\gen{a,b:a^{2^{n-1}}=1,a^{2^{n-2}}=b^2,a^b=a^{-1}}$, $n\geq 3$,
\item[(vi)]the semi-dihedral group $SD_{2^n}=\gen{a,b:a^{2^{n-1}}=b^2=1,a^b=a^{{2^{n-2}-1}}}$, $n\geq 3$.
\end{itemize}
\end{theorem}
\begin{lemma}\label{threecycles}
A non-cyclic finite group whose power graph is $K_{1,4}$-free has exactly three maximal cyclic groups.
\end{lemma}
\begin{proof}
Let $G$ be a finite group such that $\p(G)$ is $K_{1,4}$-free. Let $\gen{x_1},\ldots,\gen{x_n}$ denote all maximal cyclic subgroups of $G$. Then $\{x_1,\ldots,x_n\}$ is an independent set, which implies that the subgraph induced by $\{1,x_1,\ldots,x_n\}$ is isomorphic to $K_{1,n}$. This implies that $n\leq3$. If $n\leq2$, then $G$ is the union of at most two cyclic subgroups and hence $G$ is a cyclic group, which is a contradiction. Therefore $n=3$, as required.
\end{proof}
\begin{theorem}
If $G$ is a finite group, then $\p(G)$ is $K_{1,4}$-free if and only if $G$ is isomorphic to one of the groups $Q_8$, $\Z_2\times\Z_2$, $\Z_{p^k}$, $\Z_{pqr}$ or $\Z_{p^m q^n}$, where $\{m,n\}\cap\{0,1,2\}\neq\emptyset$ and $p,q,r$ are distinct primes.
\end{theorem}
\begin{proof}
First suppose that $G$ is a $p$-group. If $p\geq3$ and $r=m_p(G)>1$, then $G$ has an elementary abelian $p$-subgroup $H=\gen{a_1,a_2,\ldots, a_r}$ of order $p^r $. Hence $\{1,a_1,a_2,a_1a_2,a_1a_2^{-1}\}$ induces a subgraph isomorphic to $K_{1,4}$, which is a contradiction. Therefore, $G$ is cyclic whenever $p$ is odd. Now, suppose that $p=2$. If $G$ is a cyclic group, then we have nothing to prove. Thus we may assume that $G$ is not cyclic. Since $\p(G)$ is $K_{1,4}$-free, by Lemma \ref{threecycles}, $G$ has exactly three maximal cyclic subgroups, say $\gen{x}$, $\gen{y}$ and $\gen{z}$. Then $G=\gen{x}\cup\gen{y}\cup\gen{z}$ for the elements of $G$ each of which belongs to a maximal cyclic subgroup of $G$. Now, by Theorem \ref{threecover}, $|x|=|y|=|z|$ and $G/N\cong\Z_2\times\Z_2$, where $N=\gen{x}\cap\gen{y}\cap\gen{z}$. Hence, $G$ has cyclic maximal subgroups and, by Theorem \ref{maximalcycle}, we have the following cases:
\begin{itemize}
\item[(1)]$G\cong\Z_{2^{n-1}}\times\Z_2=\gen{a}\times\gen{b}=\gen{a}\cup\gen{b}\cup\gen{ab}$, where $|a|=2^{n-1}$, $|b|=2$ and $n\geq2$. Then $|a|=|b|=2$ and $G\cong\Z_2\times\Z_2$.
\item[(2)]$G\cong M_{2^n}=\gen{a,b:a^{2^{n-1}}=b^2=1,a^b=a^{1+2^{n-2}}}=\gen{a}\cup\gen{b}\cup\gen{ab}$, $n\geq 3$. Then $|a|=|b|=2$, which is a contradiction.
\item[(3)]$G\cong D_{2^n}=\gen{a,b:a^{2^{n-1}}=b^2=1, a^b=a^{-1}}=\gen{a}\cup\gen{b}\cup\gen{ab}$, $n\geq 3$. Then $|a|=|b|=2$, which is a contradiction.
\item[(4)]$G\cong Q_{2^n}=\gen{a,b:a^{2{n-1}}=1,a^{2^{n-2}}=b^2,a^b=a^{-1}}=\gen{a}\cup\gen{b}\cup\gen{ab}$, $n\geq3$. Then $2^{n-1}=|a|=|b|=4$ and $n=3$. Clearly, $\p(Q_8)$ is $K_{1,4}$-free.
\item[(5)]$G\cong SD_{2^n}=\gen{a,b:a^{2^{n-1}}=b^2=1,a^b=a^{{2^{n-2}-1}}}=\gen{a}\cup\gen{b}\cup\gen{ab}$, $n\geq 3$. Then $|a|=|b|=n=2$, which is a contradiction. 
\end{itemize} 
Therefore, $G\cong Q_8$, $\Z_2\times\Z_2$.

Now, assume that $G$ is not a $p$-group. Clearly, $|\pi(G)|\leq3$. Let $P=S_p(G)$ be a Sylow $p$-subgroup of $G$. If $P$ is not cyclic, then by Lemma \ref{threecycles}, $P$ has three maximal cyclic groups, namely $\gen{x}$, $\gen{y}$ and $\gen{z}$. If $w$ is a $q$-element ($q\neq p$), then the set $\{1,x,y,z,w\}$ induces a subgraph isomorphic to $K_{1,4}$, which is a contradiction. Thus $P=\gen{x}$ is a cyclic group. If $P$ is not a normal subgroup of $G$, then $[G:N_G(P)]\geq3$ and hence $P$ has three distinct conjugates, say $\gen{x}$, $\gen{x^g}$ and $\gen{x^h}$. Then $\{1,x,x^g,x^h,w\}$ induces a subgraph isomorphic to $K_{1,4}$, which is a contradiction. Therefore $P\normal G$, which implies that $G$ is a cyclic group. If $\pi(G)=\{p,q,r\}$, then $G\cong\Z_{p^mq^nr^k}$, where $p,q,r$ are distinct primes. If $m\geq2$, then the elements of orders $1,p^2,pq,pr,qr$ induce a subgraph isomorphic to $K_{1,4}$, a contradiction. Thus $m=1$ and similarly $n=k=1$. Then $G\cong\Z_{pqr}$. Finally, suppose that $\pi(G)=\{p,q\}$. Then $G\cong\Z_{p^mq^n}$. If $m,n\geq3$, then the elements of orders $1,p^3,p^2q,pq^2,q^3$ induce a subgraph isomorphic to $K_{1,4}$, which is impossible. Thus $m\leq2$ or $n\leq2$. The converse is obvious and the proof is complete.
\end{proof}
\section{$C_4$-free power graphs} 
In this section, we will give some structural results for a finite group to have a $C_4$-free power graph. The following key lemma will be used frequently in the sequel.
\begin{lemma}\label{C4free}
Let $G$ be a finite group. Then $\p(G)$ has an induced $4$-cycle if and only if there exist nontrivial elements $x,y$ of $G$ such that $\gen{x}\not\leq\gen{y}$, $\gen{y}\not\leq\gen{x}$ and $\gen{x}\cap\gen{y}$ is not a prime power group.
\end{lemma}
\begin{proof}
Assume that $\p(G)$ has an induced $4$-cycle $\{x,u,y,v\}$. Without loss of generality, we may suppose that $u,v\in \gen{x}\cap\gen{y}$. Since $x$ and $y$ are not adjacent, $\gen{x}\not\leq\gen{y}$ and $\gen{y}\not\leq\gen{x}$. If $\gen{x}\cap\gen{y}$ is a $p$-group, then $\gen{x}\cap\gen{y}\setminus\{1\}$ is a complete subgraph of $\p(G)$, which implies that $u\sim v$, a contradiction. Thus $\gen{x}\cap\gen{y}$ is not a $p$-group. Conversely, if $\gen{x}\cap\gen{y}$ is not a $p$-group, then we can choose elements $u,v\in\gen{x}\cap\gen{y}$ of distinct prime orders, from which it follows that $\{x,u,y,v\}$ induces a $4$-cycle, as required.
\end{proof}
\begin{corollary}
Let $G$ be finite group whose nontrivial elements have prime power orders. Then $\p(G)$ is $C_4$-free.
\end{corollary}

It is worth noticing that the structure of the groups in the above corollary has been clarified in Brandl \cite{rb} and later by others, and that it can be classified as follows.
\begin{theorem}[Bannuscher and Tiedt, {\cite[Theorem 2]{wb-gt}}]
Let $G$ be a finite group whose nontrivial elements have prime power orders. Then one of the following holds:
\begin{itemize}
\item[(1)]$G$ is isomorphic to $PSL(2,q)$ ($q=4,7,8,9,17$), $PSL(3,4)$, $Sz(8)$, $Sz(32)$ or $M_{10}$,
\item[(2)]$G$ has a nontrivial normal elementary abelian $2$-subgroup $P$ such that $\frac{G}{P}$ is isomorphic to $PSL(2,4)$, $PSL(2,8)$, $Sz(8)$ or $Sz(32)$,
\item[(3)]$G$ is a $p$-group,
\item[(4)]$G$ is a Frobenius group whose kernel is a $p$-group and its complements are either cyclic $q$-groups ($q\neq p$) or generalized quaternion $2$-groups,
\item[(5)]$G$ is a $3$-step group of order $p^aq^b$ ($p,q$ are primes and $q>2$), i.e., $G=O_{pp'p}(G)$ and $G\supset O_{p'p}(G)$ with
\begin{itemize}
\item[(i)]$O_{p'p}(G)$ is a Frobenius group with kernel $O_p(G)$ and cyclic complement, and
\item[(ii)]$G/O_p(G)$ is a Frobenius group with kernel $O_{pp'}(G)/O_p(G)$.
\end{itemize}
\end{itemize}
\end{theorem}
\begin{definition}
A collection $\Pi$ of non-trivial subgroups of a group $G$ is said to be a \textit{partition} for $G$ if every nontrivial element of $G$ belongs to exactly one subgroup in $\Pi$, that is, $G=\bigcup_{X\in\Pi}X$ and $X\cap Y=1$ for all distinct subgroups $X,Y\in\Pi$. The elements of $\Pi$ are called the \textit{components} of the partition $\Pi$.
\end{definition}
\begin{definition} 
Let $G$ be a group and $p$ be a prime. The \textit{Hughes subgroup} $H_p(G)$ of $G$, with respect to $p$, is the subgroup generated by all elements of $G$ whose orders are different from $p$.
\end{definition}
\begin{remark}
It is both easy and obvious that for any group $G$ and prime $p$, we have $H_p(H_p(G))=H_p(G)$.
\end{remark}
\begin{theorem}\label{elementform}
Let $G$ be a finite group with $C_4$-free power graph. If $x\in G$ is a non-trivial element, then $|x|=p^m$, $p^mq$ or $pqr$, where $p,q,r$ are distinct primes and $m$ is a positive integer. Moreover, 
\begin{itemize}
\item[(1)]if $|x|=pqr$, then $C_G(x)=\gen{x}$,
\item[(2)]if $|x|=p^mq$ and $m>1$, then $H_p(S_p(C_G(x)))$ is a normal cyclic subgroup of $C_G(x)$ and $\exp(S_q(C_G(x)))=q$,
\item[(3)]if $|x|=pq$, then $|C_G(x)|=p^uq^v$ or $|C_G(x)|=p^uq^vr^w$. If $|C_G(x)|=p^uq^vr^w$, then the Sylow subgroups of $C_G(x)$ have prime exponents. Furthermore, $S_r(C_G(x))$ is a normal cyclic subgroup of $C_G(x)$ and $C_G(x)=\gen{x}\times(\gen{y}\rtimes\gen{z})$, where $|y|=r$ and $|z|$ divides $pq$. If $|C_G(x)|=p^uq^v$, then $\exp(S_r(C_G(x)))=r$ and $H_s(S_s(C_G(x)))$ is a normal cyclic subgroup of $C_G(x)$ for $\{r,s\}=\{p,q\}$.
\end{itemize}
\end{theorem}
\begin{proof}
Let $x\in G$. We discuss on the number of primes dividing the order of $x$.

(i) If $|x|$ is divisible by four distinct primes $p,q,r,s$, then there exist elements $x_1, x_2, x_3, x_4$ in $\gen{x}$ whose orders are equal to $pqr,p,pqs,q$, respectively, and the subgraph induced by $x_1,x_2,x_3,x_4$ is a $4$-cycle, which is a contradiction.

(ii) Let $|x|=p^mq^nr^k$, where $m,n,k$ are natural numbers. If $|x|\neq pqr$, then we may assume that $m\geq2$. Hence there exist elements $x_1,x_2,x_3,x_4$ in $\gen{x}$ with orders $p^2q,p,pqr,q$, respectively. Then, the subgraph induced by $x_1,x_2,x_3,x_4$ is a $4$-cycle, which is a contradiction. Therefore $|x|=pqr$.

(iii) Let $|x|=p^mq^n$ ($m\geq n$), where $m,n$ are natural numbers. If $n\geq2$, then there exist elements $x_1,x_2,x_3,x_4$ in $\gen{x}$ with orders $pq^2,p,p^2q,q$, respectively. Hence, the subgraph induced by $x_1,x_2,x_3,x_4$ is a $4$-cycle, which is a contradiction. Thus $|x|=p^mq$.

Hence the order of non-trivial elements of $G$ equals $p^m$, $p^mq$ or $pqr$ for some primes $p,q,r$. In the remainder, we shall treat the structure of centralizers of elements of $G$.

First suppose that $|x|=pqr$. Clearly $\omega(C_G(x))\subseteq\{1,p,q,r,pq,pr,qr,pqr\}$. Let $g\in C_G(x)\setminus\gen{x}$ be an element of prime order $s\in\{p,q,r\}$. Then, by Lemma \ref{C4free}, $x^sg,x$ generate an induced $4$-cycle, which is a contradiction. Thus $C_G(x)$ has a unique subgroup of each prime order. On the other hand, $\exp(S_s(C_G(x))=s$, which implies that $S_s(C_G(x))\cong\Z_s$ for each $s\in\{p,q,r\}$. Hence $C_G(x)=\gen{x}$.

Now assume that $|x|=p^mq$. If $m>1$ then we have $\pi(C_G(x))=\{p,q\}$ and $\exp(S_q(C_G(x)))=q$. If $y,z\in S_p(C_G(x))$ such that $\gen{y}\nleqslant\gen{z}$, $\gen{z}\nleqslant\gen{y}$ and $\gen{y}\cap\gen{z}\neq 1$, then by Lemma \ref{C4free}, the elements $yx^{p^m}$ and $zx^{p^m}$ generate an induced $4$-cycle, which is impossible. Thus, the maximal cyclic subgroups of $S_p(C_G(x))$ form a partition $\Pi$ for $S_p(C_G(x))$. If $S_p(C_G(x))$ is non-cyclic, then by \cite{gz:2003}, $H_p(S_p(C_G(x)))\neq S_p(C_G(x))$. Since $H_p(S_p(C_G(x)))\cap\Pi$ is a partition for $H_p(S_p(C_G(x)))$, it follows that $H_p(S_p(C_G(x)))$ is cyclic for $H_p(H_p(S_p(C_G(x))))=H_p(S_p(C_G(x)))$. Put $H_p(S_p(C_G(x)))=\gen{y}$. If $\gen{y}$ is not a normal subgroup of $C_G(x)$, then $\gen{y}\neq\gen{y}^g$ for some $g\in G$. Hence, the elements $yx^{p^m}$ and $y^gx^{p^m}$ generate an induced $4$-cycle, which is a contradiction. Therefore $H_p(S_p(C_G(x)))$ is a normal cyclic subgroup of $C_G(x)$.

Finally suppose that $m=1$. Let $y,z\in C_G(x)$ be elements of distinct prime orders $r,s$ different from $p,q$, respectively. Then by Lemma \ref{C4free}, $xy$ and $xz$ generate an induced $4$-cycle, which is a contradiction. Hence $|C_G(x)|=p^uq^v$ or $p^uq^vr^w$ for some prime $r\neq p,q$. 

First, we assume that $|C_G(x)|=p^uq^vr^w$. We show that the Sylow subgroups of $C_G(x)$ have prime exponents. If $y\in C_G(x)$ is an element of order $r^2$, then $|xy|=pqr^2$, which is impossible. Thus $\exp(S_r(C_G(x)))=r$. If $y,z\in C_G(x)$ are elements with orders $r$ and $q^2$, respectively, then by Lemma \ref{C4free}, we should have $\gen{xy}\cap\gen{xz}\neq\gen{x}$. Thus $x^p\not \in \gen z$, which implies that $\gen{x,z}\cong\Z_p\times\Z_q\times\Z_{q^2}$. But $H_q(\Z_q\times\Z_{q^2})$ is not cyclic contradicting part (2). Therefore $\exp(S_q(C_G(x)))=q$ and similarly $\exp(S_p(C_G(x)))=p$. If $y,y'\in C_G(x)$ are elements of order $r$ such that $\gen{y }\neq\gen{y'}$, then by Lemma \ref{C4free}, $xy$ and $xy'$ generate an induced $4$-cycle, which is a contradiction. Therefore, $S_r(C_G(x))=\gen{y}$ is a normal subgroup of $C_G(x)$. Put $C:=C_{C_G(x)}(y)$. If $|C|>pqr$, then we may assume that $|C|=p^iq^jr$ for some $i>1$. Hence $\gen{x^q}\subset S_p(C)$. If $g\in S_p(C)\setminus\gen{x^q}$, then by Lemma \ref{C4free}, $xy$ and $x^pyg$ generate an induced $4$-cycle, which is a contradiction. Thus $|C|=pqr$. Clearly $C_G(x)/C=\gen{zC}$ is cyclic for it is isomorphic to a subgroup of $\Aut(\gen{y})$. Since $\exp(S_s(C_G(x)))=s$ for $s\in\{p,q\}$, it follows that $|zC|$ divides $pq$. Hence $z^{|zC|}=1$ if we further assume that $|z|$ is minimal subject to the given conditions. Thus $\gen{z}\cap C=1$ and $\gen{x}\cap\gen{y,z}=1$. Clearly $C=\gen{x,y}$. Hence $C_G(x)=\gen{x,y,z}=\gen{x}\times\gen{y,z}=\gen{x}\times(\gen{y}\rtimes\gen{z})$.

Now, assume that $|C_G(x)|=p^uq^v$. If $\exp(S_s(C_G(x)))=s$ for $s=p$ and $q$, then we have nothing to prove. Thus we may assume that $\exp(S_q(C_G(x)))>q$ and $y,z\in S_q(C_G(x))$ be elements with orders greater than $q$. It is easy to see, by part (2) applied to groups $\gen{x,y}$ and $\gen{x,z}$, that $x^p\in\gen{y}\cap\gen{z}$. If $y$ and $z$ are not adjacent, then by Lemma \ref{C4free}, the elements $xy$ and $xz$ induce a $4$-cycle, which is a contradiction. Thus $H_q(S_q(C_G(x)))$ is a normal cyclic subgroup of $C_G(x)$. If $\exp(S_p(C_G(x)))>p$, then similarly $H_p(S_p(C_G(x)))$ is a normal cyclic subgroup of $G$ from which it follows that $C_G(x)$ contains an element of order $p^2q^2$, which is impossible. Therefore $\exp(S_p(C_G(x)))=p$ and the proof is complete.
\end{proof}
\begin{theorem}\label{nilpotentC_4}
Let $G$ be a finite nilpotent group. Then $\p(G)$ is $C_4$-free if and only if 
\begin{itemize}
\item[(1)]$G\cong\Z_{pqr}$, 
\item[(2)]$G=P\times Q$, $H_p(P)$ is cyclic and $\exp(Q)=q$, in which $P$ is a $p$-group and $Q$ is a $q$-group, or
\item[(3)]$G$ is a $p$-group,
\end{itemize}
where $p,q,r$ are distinct primes.
\end{theorem}
\begin{proof}
Since $G$ is nilpotent, by Theorem \ref{elementform}, we have $|\pi(G)|\leqslant 3$. We have three cases:

(1) If $\pi(G)=\{p,q,r\}$, then there exists an element $z\in Z(G)$ of order $pqr$, from which, by Theorem \ref{elementform}(1), it follows that $G=C_G(z)=\gen{z}\cong\Z_{pqr}$.

(2) Suppose $\pi(G)=\{p,q\}$ and $G=P\times Q$, where $P$ is a $p$-group and $Q$ is a $q$-group. If $\exp(P)=p$ and $\exp(Q)=q$, then there is nothing to prove. Thus we may assume without loss generality in conjunction with Theorem \ref{elementform}(2) that $\exp(P)>p$ and $\exp(Q)=q$. If $P$ is cyclic then so is $H_p(P)$. If $P$ is not cyclic, then there exist elements $x,y\in P$ such that neither $\gen{x}\leq\gen{y}$ nor $\gen{y}\leq\gen{x}$. Hence, if $\gen{x}\cap \gen{y}\neq 1$, then by Lemma \ref{C4free}, the elements $xw$ and $yw$ produce an induced $4$-cycle for all nontrivial $q$-element $w$ of $Q$, which is contradiction. Thus the maximal cyclic subgroups of $P$ partition $P$ and, by the same argument as before, $H_p(P)$ is cyclic.

(3) there is nothing to prove.

Conversely, if $G\cong\Z_{pqr}$ or $G$ is finite $p$-group, then clearly $\p(G)$ is $C_4$-free. Now, suppose that $G$ is a group as in part (2). Suppose on the contrary that $\p(G)$ has an induced $4$-cycle. Then there exist elements $x,y$ such that neither $\gen{x}\leq\gen{y}$ nor $\gen{y}\leq\gen{x}$, and $\gen{x}\cap\gen{y}$ is not a group of prime power order. Let $\gen{x}\cap\gen{y}=\gen{a}\times\gen{b}$, where $a$ is a $p$-element and $b$ is a $q$-element. Then $x=cb$, $y=db$ and $\gen{c}\cap\gen{d}=\gen{a}$, where $c,d$ are $p$-elements. If $H_p(P)=P$, then $P$ is cyclic and either $\gen{x}\leq\gen{y}$ or $\gen{y}\leq\gen{x}$, which is a contradiction. Also, if $H_p(P)\neq P$, then $P$ has a non-trivial partition with cyclic components, which implies that either $\gen{x}\leq\gen{y}$ or $\gen{y}\leq\gen{x}$, which is another contradiction. The proof is complete.
\end{proof}
\subsection{Groups with prescribed centers}
Utilizing Theorems \ref{elementform} and \ref{nilpotentC_4}, we can give further results for the groups under investigation when the center of group is divisible by at least two primes.
\begin{theorem}
Let $G$ be a group with $C_4$-free power graph. If $Z(G)$ is not a $p$-group, then
\begin{itemize}
\item[(1)]$G\cong\Z_{pqr}$,
\item[(2)]$G=P\times Q$ or $G=(P\times Q)\rtimes\Z_q$, where $P$ is cyclic, $\exp(Q)=q$ and $C_G(P)=P\times Q$,
\item[(3)]$G=D_{2^n}\rtimes Q$, where $\exp(Q)=q$,
\item[(4)]$G=\Z_{p^n}\rtimes Q$, where $\exp(Q)=q$,
\item[(5)]$G=\Z_{pq}\times(\Z_r\rtimes\Z_p)$, where $C_G(S_r(G))\cong\Z_{pqr}$, or
\item[(6)]$G=\Z_{pq}\times(\Z_r\rtimes\Z_{pq})$, where $C_G(S_r(G))\cong\Z_{pqr}$.
\end{itemize}
\end{theorem}
\begin{proof}
If $|\pi(Z(G))|\geq3$, then by Theorem \ref{elementform}, $G\cong\Z_{pqr}$ for some distinct primes $p,q$ and $r$. 

Now, suppose that $\pi(Z(G))=\{p,q\}$. Then $Z(G)=P\times Q$, where $P$ is a $p$-group and $Q$ is a $q$-group. We have two cases: 

Case 1: $\exp(P)>p$. By Theorem \ref{elementform}(2), $\pi(G)=\{p,q\}$. Also, by Theorem \ref{nilpotentC_4}(2), $P$ is cyclic and $\exp(Q)=q$. If $x\in S_p(G)$, then $\gen{Z(G),x}=\gen{P,x}\times Q$. By Theorem \ref{nilpotentC_4}(2), $\gen{P,x}$ is cyclic, which implies that $G$ has a unique subgroup of order $p$. Thus $S_p(G)$ is either a cyclic group or a generalized quaternion $2$-group. If $S_p(G)$ is a generalized quaternion $2$-group, there there exists $x\in S_p(G)\setminus P$ such that $|x|=4$. Since $\gen{P,x}$ is cyclic and $\exp(P)\geq4$ it follows that $x\in P$, which is a contradiction. Thus $S_p(G)$ is cyclic. Put $S_p(G)=\gen{x}$. If $\gen{x}\not\normal G$, then $\gen{x}\neq\gen{x}^g$ for some $g\in G$. But then by Lemma \ref{C4free}, $xy$ and $x^gy$ give rise to an induced $4$-cycle for all $y\in Q\setminus\{1\}$, which is a contradiction. Hence $\gen{x}\normal G$. Since $G/C_G(x)$ is a $q$-group of exponent dividing $q$ and it is isomorphic to a subgroup of $\Aut(\gen{x})$, it follows that $G=C_G(x)$ or $G/C_G(x)\cong\Z_q$, from which part (2) follows.

Case 2: $\exp(P)=p$ and $\exp(Q)=q$. If $r\in\pi(G)\setminus\{p,q\}$, then by Theorem \ref{elementform}, $\exp(S_r(G))=r$. Assume $\exp(S_p(G))>p$ and $x,y$ be $p$-elements of the same order $p^m>p$. Since $\gen{Z(G),x}=\gen{P,x}\times Q$, by Theorem \ref{nilpotentC_4}(2), $\gen{P,x}$ is cyclic, hence $P\cong\Z_p$ and $P\subseteq\gen{x}$. Similarly, $P\subseteq\gen{y}$. If $\gen{x}\neq\gen{y}$, then by Lemma \ref{C4free}, $xz$ and $yz$ produce an induced $4$-cycle for all $z\in Q\setminus\{1\}$, which is a contradiction. Thus $G$ has a unique cyclic $p$-subgroup of order $p^m$. In particular $\gen{x}\normal G$. Since by Theorem \ref{elementform}, $G$ has no elements of order $p^2q^2$, it follows that $\exp(S_q(G))=q$. If $w\in G$ is an element of order $r\in\pi(G)\setminus\{p,q\}$, then by Lemma \ref{C4free}, the elements $xz$ and $x^{|x|/p}zw$ give rise to an induced $4$-cycle for all $z\in Q\setminus\{1\}$, which is a contradiction. Thus $\pi(G)=\{p,q\}$. Now, suppose that the order of $x$ is maximal with respect to being a $p$-element. Clearly, $C_{S_p(G)}(x)=\gen{x}$. If $p=2$, then since $S_p(G)\setminus\gen{x}$ contains only involutions acting on $\gen{x}$ by inversion, it follows that $[S_p(G):\gen{x}]\leq2$. If $p>2$, then since $S_p(G)/\gen{x}$ is isomorphic to a $p$-subgroup of $\Aut(\gen{x})$ of exponent at most $p$, it follows that $[S_p(G):\gen{x}]\leq p$. By Theorem \cite[5.3.4]{djsr:1996}, $S_p(G)$ is a cyclic group or a dihedral $2$-group, from which we obtain parts (3) and (4), respectively. 

Finally, suppose that $\exp(S_p(G))=p$, $\exp(S_q(G))=q$, and $x,y$ are central elements of orders $p$ and $q$, respectively. If $\pi(G)=\{p,q\}$, then we are done. Thus we may assume that $\pi(G)\supset\{p,q\}$. Clearly, $\exp(S_r(G))=r$ for all $r\in\pi(G)\setminus\{p,q\}$. If $r,s\in\pi(G)\setminus\{p,q\}$, $a$ is an element of order $r$, $b$ is an element of order $s$ and $\gen{a}\neq\gen{b}$, then by Lemma \ref{C4free}, $xya$ and $xyb$ produce an induced $4$-cycle, which is a contradiction. Hence, $\pi(G)=\{p,q,r\}$ for some $r$ and $S_r(G)=\gen{z}$ is a normal cyclic subgoup of $G$ of order $r$. By Theorem \ref{elementform}, $C_G(z)=C_G(xyz)=\gen{xyz}$. Since $G/C_G(z)$ is isomorphic to a subgroup of $\Aut(\gen{z})$, it follows that $G/C_G(z)$ is a cyclic gorup of order $p$, $q$ or $pq$, from which we obtain parts (5) and (6). The proof is complete.
\end{proof}
\begin{theorem}
Let $G$ be a finite group with $C_4$-free power graph, which is not a prime power group. If $Z(G)$ is a $p$-group which is not an elementary abelian $p$-group, then $Z(G)$ is cyclic and $\exp(S_q(G))=q$ for every $q\neq p$. Also, 
\begin{itemize}
\item[(1)]$\pi(C_G(x))=\{p,q\}$, 
\item[(2)]$S_p(C_G(x))$ is a normal cyclic subgroup of $C_G(x)$,
\item[(3)]$C_G(x)=\gen{y}\times S_q(C_G(x))$, or $C_G(x)=(\gen{y}\times S_q(C_{C_G(x)}(y)))\rtimes\Z_q$ if $p$ is odd prime, where $\gen{y}=S_p(C_G(x))$,
\end{itemize} 
for every $q$-element $x\in G$ such that $q\neq p$.
\end{theorem} 
\begin{proof}
If $x\in G$ is a $q$-element ($q\neq p$), then by Theorem \ref{nilpotentC_4}, $\gen{Z(G), x}\cong \Z_{p^m}\times\Z_q$, where $|Z(G)|=p^m$ ($m>1$). Hence $Z(G)$ is cyclic and $\exp(S_q(G))=q$ for all $q\neq p$. Put $Z(G)=\gen{z}$.

If $y\in C_G(x)\setminus\gen{z}$ is an element of prime order $p$, then by Lemma \ref{C4free}, the elements $xz,xyz$ produce an induce $4$-cycle, which is a contradiction. Thus $S_p(C_G(x))$ has a unique subgroup of prime order $p$ and hence, by \cite[5.3.6]{djsr:1996}, $S_p(C_G(x))$ is cyclic or a generalized quaternion $2$-group. If $S_2(C_G(x))$ is a generalized quaternion $2$-group, we can choose an element $y\in C_G(x)$ with order $4$ such that neither $\gen{y}\leq Z(G)$ nor $Z(G)\leq\gen{y}$ and, by Lemma \ref{C4free}, the elements $xy,xz$ give rise to an induced $4$-cycle, which is a contradiction. Hence $S_p(C_G(x))$ is cyclic and by the same argument we can show that $S_p(C_G(x))$ is normal in $C_G(x)$.

Put $S_p(C_G(x))=\gen{y}$ and $C:=C_{C_G(x)}(y)$. Since $\gen{y}\normal C_G(x)$, $C_G(x)/C$ is isomorphic to a subgroup of $\Aut(\gen{y})$. On the other hand, $\Aut(\gen{y})\cong\Z_{p^k(p-1)}$ or $\Z_{2^k}\times\Z_2$, where $p^k=|y|$. If $p=2$ then $C_G(x)/C$ is the trivial group, otherwise since $C_G(x)/C$ is a $q$-group, $C_G(x)/C$ is a cyclic group of order dividing $p-1$. As $\exp(S_q(G))=q$, the group $C_G(x)/C$ is trivial or a cyclic group of order $q$. Hence $C_G(x)=\gen{y}\times S_q(C_G(x))$ or $C_G(x)=(\gen{y}\times S_q(C_{C_G(x)}(y)))\rtimes\Z_q$. The proof is complete.
\end{proof}
\begin{theorem}
Let $G$ be a finite group with $C_4$-free power graph, which is not a prime power group. If $Z(G)$ is an elementary abelian $p$-group of order $>p$, then for every $q$-element $x$ ($q\neq p$), we have $\pi(C_G(x))=\{p,q\}$ and if $|x|>q$, then
\begin{itemize}
\item[(1)]$\exp(S_p(C_G(x)))=p$ and $S_q(C_G(x))$ is a normal cyclic subgroup of $C_G(x)$,
\item[(2)]$C_G(x)=S_p(C_G(x))\times S_q(C_G(x))$ or $(S_p(C_{C_G(x)}(y))\times S_q(C_G(x)))\rtimes\Z_p$, where $\gen{y}=S_q(C_G(x))$.
\end{itemize}
\end{theorem}
\begin{proof}
Let $x$ be a $q$-element ($q\neq p$). If $\pi(C_G(x))\neq\{p,q\}$, then there exists $y\in C_G(x)$ such that $|y|=r\neq p,q$, from which by utilizing Lemma \ref{C4free}, it follows that $xyz_1,xyz_2$ generate an induced $4$-cycle for all $z_1,z_2\in Z(G)\setminus\{1\}$ such that $\gen{z_1}\neq\gen{z_2}$, which is a contradiction. Thus $\pi(C_G(x))=\{p,q\}$.

Now, suppose that $|x|>q$. Then by Theorem \ref{nilpotentC_4} applied to the group $Z(G)\times S_q(C_G(x))$,  $H_q(S_q(C_G(x)))$ is cyclic. If $H_q(S_q(C_G(x)))\neq S_q(C_G(x))$ and $y\in S_q(C_G(x))\setminus H_q(S_q(C_G(x)))$, then $xy\in S_q(C_G(x))\setminus H_q(S_q(C_G(x)))$, which implies that $|x|=|xy|=q$, a contradiction. Thus $H_q(S_q(C_G(x)))=S_q(C_G(x))$, which implies that $S_q(C_G(x))=\gen{y}$ is a cyclic gorup. If $S_q(C_G(x))\not\normal C_G(x)$, then there exists $g\in C_G(x)$ such that $\gen{y}\neq\gen{y^g}$. Hence, by Lemma \ref{C4free}, the elements $yz$ and $y^gz$ produce an induced $4$-cycle for all $z\in Z(G)\setminus\{1\}$, which is a contradiction. Thus $S_q(C_G(x))\normal C_G(x)$. Clearly, by Theorem \ref{elementform}, $\exp(S_p(C_G(x)))=p$. Put $C:=C_{C_G(x)}(y)$. Since $C_G(x)/C$ is $p$-group of exponent $p$ isomorphic to a subgroup of $\Aut(\gen{y})\cong\Z_{q-1}$, we must have $C_G(x)=C$ or $C_G(x)/C\cong\Z_p$, from which the result follows.
\end{proof}

\end{document}